\documentclass{amsart}%
\title{Optimal eigenvalue estimates for the Robin Laplacian on Riemannian manifolds}
\author{Alessandro Savo}

\usepackage{amssymb , amsmath, amsthm}
\newtheorem{defi}{Definition} 

\newtheorem{theorem}[defi]{Theorem}

\newtheorem{lemme}[defi]{Lemma}
\newtheorem{lemma}[defi]{Lemma}
\newtheorem{cor}[defi]{Corollary}

 \newcommand{\twosystem}[2]{\left\{\begin{aligned} &#1\\ &#2\end{aligned}\right.}
\newcommand{\threesystem}[3]{\left\{ \begin{aligned}&#1\\ &#2\\&#3\end{aligned}\right.}

\newcommand{\threearray}[3]{\begin{aligned}&{#1}\\&{#2}\\&{#3}\end{aligned}}

\newcommand{\nero}{\smallskip$\bullet\quad$\rm}

\newcommand{\parte}[1]{\smallskip\noindent {\rm#1)}\,\,}
\newcommand{\scal}[2]{\langle{#1},{#2}\rangle}

\newcommand{\abs}[1]{\lvert{#1}\rvert}

\newcommand{\reals}{{\bf R}}

\newcommand{\sphere}[1]{{\bf S}^{#1}}
\newcommand{\real}[1]{{\bf R}^{#1}}

\newcommand{\bd}{\partial}

\newcommand{\derive}[2]{\dfrac{\bd #1}{\bd#2}}

\newcommand{\matrice}{\begin{pmatrix}}
\newcommand{\ok}{\end{pmatrix}}

\def\<{\left\langle}
\def\>{\right\rangle}

\begin{document}

\footnote{Research partially supported by Istituto Nazionale di Alta Matematica and GNSAGA of Italy.}

\address{Alessandro Savo, {\rm Dipartimento SBAI, Sezione di Matematica,
Sapienza Universit\`a di Roma,
Via Antonio Scarpa 16,
00161 Roma, Italy}}
 
 \email{alessandro.savo@uniroma1.it}

\maketitle
\large

\begin{abstract} We consider the first eigenvalue $\lambda_1(\Omega,\sigma)$ of the Laplacian with Robin boundary conditions on a compact Riemannian manifold $\Omega$ with smooth boundary, $\sigma\in\reals$ being the Robin boundary  parameter. 
When $\sigma>0$ we give a positive, sharp lower bound of $\lambda_1(\Omega,\sigma)$ in terms of an associated one-dimensional problem depending on the geometry through  a lower bound of the Ricci curvature of $\Omega$, a lower bound of the mean curvature of $\bd\Omega$ and the inradius.  When the boundary parameter is negative, the lower bound becomes an upper bound. 
In particular, explicit bounds for mean-convex Euclidean domains are obtained, which improve known estimates. 

Then, we extend a monotonicity result for $\lambda_1(\Omega,\sigma)$ obtained in  Euclidean space by Giorgi and Smits \cite{GS}, to a class of manifolds of revolution which include all space forms of constant sectional curvature. As an application, we prove that $\lambda_1(\Omega,\sigma)$ is uniformly bounded below by 
$\frac{(n-1)^2}4$ for all bounded domains in the hyperbolic space of dimension $n$, provided that the boundary parameter $\sigma\geq\frac{n-1}{2}$ (McKean-type inequality). Asymptotics for large hyperbolic balls are also discussed. \footnote

\end{abstract}

\large

\section{Introduction} 

\subsection{Definition and some known facts} Let $(\Omega^n,g)$ be a compact Riemannian manifold of dimension $n$ with smooth boundary $\bd\Omega$, and let $\Delta$ be the Laplacian associated to the metric $g$. The sign convention is that, on $\real n$:
$$
\Delta=-\sum_{k=1}^n\dfrac{\bd^2}{\bd x_k^2}.
$$
 We are interested in the first eigenvalue of the Robin problem:
\begin{equation}\label{robinproblem}
\twosystem
{\Delta u=\lambda u\quad\text{on}\quad\Omega}
{\derive uN=\sigma u\quad\text{on}\quad\bd\Omega}
\end{equation}
where $\sigma\in\reals$ is a parameter and $N$ is the inner unit normal.
The eigenvalues form a discrete sequence diverging to infinity:
$$
\lambda_1(\Omega, \sigma)<\lambda_2(\Omega,\sigma)\leq\dots;
$$
it is known that the first eigenvalue $\lambda_1(\Omega,\sigma)$ is simple and that any first eigenfunction does not change sign, so that we can assume that it is positive. 
 
\smallskip

The first eigenvalue $\lambda_1(\Omega,\sigma)$ models heat diffusion with absorbing ($\sigma>0$) or radiating ($\sigma<0$) boundary; it can also be seen as the fundamental tone of an elastically supported membrane. 

\smallskip

\subsection{Some features  of $\lambda_1(\Omega,\sigma)$} It is immediately seen that for $\sigma=0$ we recover the classical Neumann problem; in particular $\lambda_1(\Omega,0)=0$, the associated eigenfunctions being the constants. 
Hence we can assume $\sigma\ne 0$. The Rayleigh min-max principle reads:
\begin{equation}\label{minmax}
\lambda_1(\Omega,\sigma)=\inf_{u\in H^1(\Omega)}\dfrac
{\int_{\Omega}\abs{\nabla u}^2+\int_{\bd\Omega}\sigma u^2}{\int_{\Omega}u^2}
\end{equation}
and one can see easily that $\lambda_1(\Omega,\sigma)$ is positive for all $\sigma>0$, and negative when $\sigma<0$. 
Moreover $\lambda_1(\Omega,\sigma)$ is an increasing function of $\sigma$ which tends to $\lambda_1(\Omega,\infty)\doteq\lambda_1^D(\Omega)$ (the first Dirichlet eigenvalue of $\Omega$) when $\sigma\to +\infty$. In particular:
$$
\lambda_1(\Omega,\sigma)<\lambda_1^D(\Omega)
$$
for all $\sigma\in\reals$. 

\nero Problem \eqref{robinproblem} continues to make sense and admits a discrete spectrum when $\sigma$ is a continuous function on $\bd\Omega$, and not just a constant. The min-max principle \eqref{minmax} makes clear that then:
$$
\lambda_1(\Omega, \inf_{\bd\Omega}\sigma)\leq \lambda_1(\Omega, \sigma)\leq \lambda_1(\Omega, \sup_{\bd\Omega}\sigma).
$$
However in this paper we tacitly assume that $\sigma$ is a real constant.

\smallskip

The behaviour when the boundary parameter $\sigma\to -\infty$ is quite interesting. In that case the first eigenfunction concentrates near the boundary and, for domains in $\real n$ having $C^{\infty}$-smooth boundary, one has the following asymptotic expansion as $\sigma\to -\infty$:
$$
\lambda_1(\Omega,\sigma)= -\sigma^2+(n-1)H_{\rm max}\sigma+o(\sqrt{\sigma}),
$$
where $H_{\rm max}$ denotes the maximum value of the mean curvature of $\bd\Omega$. This was first proved by Pankrashkin \cite{P}  in dimension $2$ and later generalized by Pankrashkin and Popoff  \cite{PP}. The presence of corners affects the first term (see \cite{LP}). 

\smallskip

Domain monotonicity is an essential feature of the Dirichlet problem: if $\Omega_1\subseteq\Omega_2$ then $\lambda_1^D(\Omega_1)\geq \lambda_1^D(\Omega_2)$. This is an important tool in estimating eigenvalues. As observed in \cite{GS} domain monotonicity does not hold in full generality when $\sigma<+\infty$, even for  convex  domains. However, it does hold in Euclidean space $\real n$ when the outer domain is a ball (see \cite{GS} Theorem 1). We will in fact extend the argument in \cite{GS}  to prove a similar monotonicity result for a certain class of revolution manifolds, in particular, for the other space forms 
${\bf H}^n$ and $\sphere n$. This will be applied to generalize the classical McKean inequality \cite{McK} to the Robin Laplacian.

\subsection{Some known eigenvalue estimates} When $\sigma>0$, a Faber-Krahn type inequality has been proved by Bossel \cite{Bos} for domains in $\real 2$;  the result was extended to domains in $\real n$ by Daners \cite {D}. The conclusion is that, among all Euclidean domains with fixed volume, the ball minimises $\lambda_1(\Omega,\sigma)$ for any fixed $\sigma >0$. 

When $\sigma<0$, it was conjectured by Bareket \cite{Bar} that the ball would be, instead, a maximiser. As shown in \cite{FNT} this is true for domains which are close to a ball. But in fact the Bareket conjecture is false in general, as shown by Freitas and Krejcirik \cite{FK},  who showed that when $\sigma$ is large negative annuli with the same volume have larger first eigenvalue. In dimension $2$, they also showed that there exists a critical parameter $\sigma^{\star}<0$, depending only on the area,  such  that for any $\sigma\in [\sigma^{\star},0]$ the ball is a maximiser. 

Finally, let us mention the following explicit upper and lower bounds for convex domains in $\real n$, proved by Kovarik in \cite{Kov} :
\begin{equation}\label{kovarikintro}
\dfrac{\sigma}{4R+4R^2\sigma}\leq\lambda_1(\Omega,\sigma)\leq \dfrac{2K_n\sigma}{R+R^2\sigma},
\end{equation}
where $K_n$ is an explicit constant. Here $R$ is the {\it inradius} of $\Omega$, that is, the largest radius of a ball included in $\Omega$. Note that the lower bound in \eqref{kovarikintro} is sometimes much better than the Faber Krahn inequality (think of a convex domain which has small inradius and fixed volume). By passing to the limit as $\sigma\to\infty$ the lower bound becomes:
\begin{equation}\label{kovarikd}
\lambda_1^D(\Omega)\geq\dfrac{1}{4R^2};
\end{equation}
the author observes in Remark 4.6 of  \cite{Kov} that, due to the method used (Hardy inequality), \eqref{kovarikd} cannot be sharp; in fact the sharp bound in terms of the inradius would be
\begin{equation}\label{liyau}
\lambda_1^D(\Omega)\geq\dfrac{\pi^2}{4R^2},
\end{equation}
as proved by Hersch in \cite{H}.  The lower bound \eqref{liyau} was later shown to hold for a wider class of Riemannian manifolds in \cite{LY} and \cite{K2} (and later by the author in \cite{S2}, by different methods).  We will in fact prove a sharp lower bound (resp. upper bound) for all $\sigma>0$ (resp. $\sigma<0$) in the Riemannian case, by adapting the method of Laplacian comparison  to the Robin boundary conditions; when applied to convex Euclidean domains, this will improve the lower bound in \eqref{kovarikintro} and yield in the limit the sharp estimate \eqref{liyau}.

\smallskip

The scope of this paper is twofold: we first  prove a comparison theorem for a general Riemannian manifold (see Theorem \ref{compintro} and Theorem 
\ref{zerocurvature}) and then we prove a monotonicity result for a large class of revolution manifolds (Theorem \ref{dmintro}).  Both these methods will produce sharp bounds (in particular, a McKean-type inequality, Theorem \ref{mckeanintro}). 

\smallskip

The paper is structured as follows. In Section \ref{main} we state our main results and in Section \ref{pre} we prove some preliminary facts. 
Section \ref{scomparison} is devoted to the proof of the comparison theorem, while in Sections \ref{monotonicity} and \ref{mckean} we prove domain monotonicity and the McKean-type inequality. Finally in the Appendix we describe the model domains for our comparison theorem. 


\section{Main results}\label{main}

\subsection{Comparison theorem} We will compare $\lambda_1(\Omega,\sigma)$ with the first eigenvalue of a one-dimensional problem on the interval $[0,R]$, where $R$ is the inradius of $\Omega$:
\begin{equation}\label{onedintro}
\threesystem
{u''+\dfrac{\Theta'}{\Theta}u'+\lambda u=0}
{u'(0)=\sigma u(0)}
{u'(R)=0}
\end{equation}
This problem carries a weight $\Theta=\Theta(r)$ depending explicitly on the geometry of $\Omega$, as follows. We say that $\Omega$ has {\it curvature data $(K,H)$} if:

\nero the Ricci curvature of $\Omega$ is bounded below by $(n-1)K$,

\nero the mean curvature of $\bd\Omega$ is bounded below by $H$.

\smallskip

We stress that $K$ and $H$ may assume any real value. Our convention on the mean curvature is the following. Let $S$ be the shape operator of the immersion of $\bd\Omega$ into $\Omega$, with respect to the inner unit normal $N$: this is the self-adjoint operator acting on $T\bd\Omega$ and defined by
$
S(X)=-\nabla_XN,
$
for all $X\in T\bd\Omega$. Then the mean curvature is 
$$
\mathcal H=\dfrac{1}{n-1}{\rm tr}S.
$$
The sign convention is such that $\mathcal H$ is positive, and equal to $\frac 1R$, on the boundary of the  ball of radius $R$ in $\real n$.
As usual, we denote by $R$ the inradius of $\Omega$. Introduce the function on $[0,R]$:
\begin{equation}\label{sk}
s_k(r)=\threesystem
{\dfrac{1}{\sqrt K}\sin(r\sqrt K), \quad\text{if $K>0$},}
{r\quad\text{if $K=0$},}
{\dfrac{1}{\sqrt {\abs K}}\sin(r\sqrt {\abs K}), \quad\text{if $K<0$}.}
\end{equation}
We now define what we will call the {\it weight function} $\Theta:[0,R]\to \reals$ by:
\begin{equation}\label{theta}
\Theta(r)=\Big(s'_K(r)-Hs_k(r)\Big)^{n-1}.
\end{equation}
Note that $\Theta$ depends on $K$ and $H$, and that $\Theta(0)=1$. As a consequence of Theorem A in \cite{K} (see also Proposition 14 in \cite{S2})
we have that $\Theta$ is positive on $[0,R)$, and moreover $\Theta(R)=0$ if and only if $\Omega$ is a geodesic ball in the space form $M_K$, that is, the simply connected manifold with constant sectional curvature $K$. 

Here is a general comparison theorem. 

\begin{theorem}\label{compintro} Let $\Omega$ be a compact manifold with smooth boundary having curvature data $(K,H)$ and inradius $R$.  If $\sigma>0$, then:
$$
\lambda_1(\Omega,\sigma)\geq \lambda_1(R,\Theta,\sigma),
$$
where $ \lambda_1(R,\Theta,\sigma)$ is the first eigenvalue of problem \eqref{onedintro}, and $\Theta$ is defined in \eqref{theta}. If $\sigma<0$ the inequality is reversed:
$$
\lambda_1(\Omega,\sigma)\leq \lambda_1(R,\Theta,\sigma).
$$
\end{theorem}

In other words one has, for all $\sigma\in\reals$:
$$
\abs{\lambda_1(\Omega,\sigma)}\geq \abs{\lambda_1(R,\Theta,\sigma)}.
$$

For the proof, see section \ref{scomparison}. The estimate is sharp in every dimension: see section \ref{sharpness} below. 
The eigenvalue $\lambda_1(R,\Theta,\sigma)$ is always positive   and,
when $\sigma>0$, the theorem gives a positive lower bound for any compact Riemannian manifold with boundary. For the Dirichlet problem ($\sigma=+\infty$) the result is due to Kasue \cite{K2}.

A particularly simple situation is when $K=H=0$, so that $\Theta(r)=1$. We obtain the following fact.

\begin{theorem} \label{zerocurvature} Let $\sigma>0$. Assume that both the Ricci curvature of $\Omega$ and the mean curvature of $\bd\Omega$ are non-negative. Let $R$ be the inradius of $\Omega$. Then:
$$
\lambda_1(\Omega,\sigma)\geq \lambda_1([0,2R],\sigma),
$$
where on the right we have the first Robin eigenvalue of the interval $[0,2R]$.
The estimate is sharp in any dimension. Precisely, if $\Omega$ is any flat cylinder (that is, a Riemannian product  $ [0,2R]\times \Sigma^{n-1}$, where $\Sigma^{n-1}$ is a closed Riemannian manifold of dimension $n-1$) then equality holds.

\smallskip

If $\sigma<0$ the inequality is reversed and sharp as well. 
\end{theorem}

Explicit evaluation of the right-hand side using the Becker-Starck inequality implies the following  estimate. 

\begin{cor}\label{convex} If both the Ricci curvature of $\Omega$ and the mean curvature of $\bd\Omega$ are non-negative (in particular, for mean-convex Euclidean domains) we have, if $\sigma>0$:
$$
\lambda_1(\Omega,\sigma)>
\dfrac{\pi^2\sigma}{\pi^2R+4R^2\sigma},
$$
while if $\sigma<0$, then $\lambda_1(\Omega,\sigma)< -\sigma^2$.
\end{cor}

(The proof is given in section \ref{proofconvex}). The estimate applies to any mean-convex (in particular, convex) domain in $\real n$, and it improves the bound \eqref{kovarikintro} for all $\sigma>0$ and $R$. As $\sigma\to +\infty$ it gives the expected sharp bound:
$$
\lambda_1^D(\Omega)\geq\dfrac{\pi^2}{4R^2}.
$$


\subsection{Sharpness, method of proof}\label{sharpness} Theorem \ref{compintro} is sharp in all dimensions. In fact, for any $R>0$ and  for any curvature data $(K,H)$ we will construct a {\it model domain} $\bar\Omega\doteq\bar\Omega(K,H,R)$ of dimension $n$ with two boundary components: $\bd\bar\Omega=\Gamma_1\cup\Gamma_2$ such that :
$$
\lambda_1(\bar\Omega,\sigma)=\lambda_1(R,\Theta,\sigma)
$$
where on the left we have the first eigenvalue of $\bar\Omega$  with Robin conditions on $\Gamma_1$ and Neumann conditions on $\Gamma_2$.
For the definition of $\bar\Omega$ we refer to the Appendix. 

\smallskip

In some cases the model domain can be a ball in a space-form $M_K$ and we have an equality case:

\begin{theorem}\label{equalitycase} Let $\Omega$ be a domain with curvature data $(K,H)$, and assume one of the following three cases: a) $K>0$ and $H\in\reals$, b) $K=0$ and $H>0$, c) $K<0$ and $H>\sqrt{\abs K}$. Then:

\parte a There is a unique ball $\tilde\Omega$ in $M_K$ with mean curvature equal to $H$.

\parte b The radius $\tilde R$ of $\tilde\Omega$ satisfies $\tilde R\geq R$.

\parte c One has $\lambda_1(\Omega,\sigma)\geq\lambda_1(\tilde\Omega,\sigma)$ with equality if and only if $\Omega$ is isometric to $\tilde\Omega$.

\end{theorem} 
For the proof, see the Appendix. 
The proof of Theorem \ref{compintro} is by Laplacian comparison, and is obtained by extending the methods in \cite{K2} and \cite{S2} to the Robin Laplacian. 


\subsection{Domain monotonicity on revolution manifolds}  As remarked before, domain monotonicity does not hold in full generality for the Robin Laplacian, and the first aim is here to extend the monotonicity result of \cite{GS} from Euclidean space to other manifolds. 

We focus here on the class of {\it revolution manifolds with pole $x_0$}: these are manifolds  $M\doteq [0,T]\times \sphere{n-1}$ ($T$ could be $+\infty$) with metric 
$$
g=dr^2+\Phi(r)^2g_{\sphere{n-1}},
$$
where $g_{\sphere{n-1}}$ is the canonical metric on the sphere $\sphere{n-1}$. Here $r\in [0,T]$ is the geodesic distance to the pole and $\Phi(r)$ is the {\it warping function}; this is a smooth, positive function on $[0,T]$ satisfying the conditions:
$$
\Phi(0)=\Phi''(0)=0, \quad \Phi'(0)=1,
$$
which are imposed in order to have a $C^2$-metric at the pole. However, if we make the stronger assumptions that $\Phi$ has vanishing even derivates at zero then the metric is $C^{\infty}$-smooth everywhere. 

\nero The geodesic ball centered at the pole of $M$ having radius $R\leq T$, that is, $B(x_0,R)$, is evidently a revolution manifold itself and will then be regarded as such. We also remark that the function
$$
H(r)\doteq\dfrac{\Phi'}{\Phi}(r)
$$
expresses the mean curvature of the geodesic sphere $\bd B(x_0,r)$ with respect to the inner unit normal $N=-\nabla r$.

\smallskip

Recall that the {\it space-form of constant curvature $K$}, denoted $M_K$, is : Euclidean space if $K=0$, the round sphere of radius $\frac{1}{\sqrt K}$ if $K>0$, and the hyperbolic space ${\bf H}^n_K$ of (constant) curvature $K$ if $K<0$. The space-form $M_K$ is a revolution manifold around any of its points, the warping function being $\Phi(r)=s_K(r)$. Precisely:
$$
\threearray
{\Phi(r)=r\quad\text{if}\quad K=0}
{\Phi(r)=\dfrac{1}{\sqrt{K}}\sin (r\sqrt{K})\quad\text{if}\quad K>0,}
{\Phi(r)=\dfrac{1}{\sqrt{-K}}\sinh (r\sqrt{-K})\quad\text{if}\quad K<0.}
$$

We will be interested in the situation where the warping function is log-concave, that is $(\log\Phi)''<0$. This is equivalent to asking that the mean curvature of $\bd B(x_0,r)$ is a decreasing function of $r$ (the distance to the pole).
 It is clear that the condition is satisfied by the warping function of all space-forms $M_K$.

\begin{theorem}\label{dmintro} Let $\Omega$ be a domain of a revolution manifold $M$ with pole $x_0$, whose warping function $\Phi$ is log-concave : $(\log\Phi)''<0$.   Assume that $\Omega\subseteq B(x_0,R)$. If $\sigma>0$ then
$$
\lambda_1(\Omega,\sigma)\geq \lambda_1(B(x_0,R),\sigma),
$$
while if $\sigma<0$ then the opposite inequality holds:
$$
\lambda_1(\Omega,\sigma)\leq \lambda_1(B(x_0,R),\sigma).
$$
In particular, the above monotonicity holds true in any space-form $M_K$.
\end{theorem}

 For the proof, see section \ref{monotonicity}. When $M=\real n$ the result is due to Giorgi and Smits \cite{GS}.


\subsection{McKean-type inequality}
It is well-known that the first Dirichlet eigenvalue of any bounded domain in ${\bf H}^n$ satisfies the bound
$$
\lambda_1^D(\Omega)\geq \dfrac{(n-1)^2}{4},
$$
known as McKean inequality (see \cite{McK}). The remarkable fact here is that the inequality holds regardless of the size of $\Omega$ (volume, diameter, etc.). 
By domain monotonicity, which is valid in ${\bf H}^n$ thanks to Theorem \ref{dmintro}, and by explicit calculations for geodesic balls, we can extend McKean inequality to the Robin Laplacian, as follows. 

\begin{theorem}\label{mckeanintro} Let $\Omega$ be a domain in ${\bf H}^n$ and let $\sigma>0$. Then:
$$
\lambda_1(\Omega,\sigma)\geq\twosystem{\dfrac{(n-1)^2}{4}\quad\text{if}\quad \sigma\geq \dfrac{n-1}{2}}
{(n-1)\sigma-\sigma^2\quad\text{if}\quad 0<\sigma\leq\dfrac{n-1}{2}}
$$
If instead we assume $\sigma<0$, then
$
\lambda_1(\Omega,\sigma)\leq -\sigma^2+(n-1)\sigma.
$
\end{theorem}
For $\sigma>\frac{n-1}2$, the estimate is sharp, because if $B_R$ is any hyperbolic ball of radius $R$, we have:
$$
\lim_{R\to\infty}\lambda_1(B_R,\sigma)=\frac{(n-1)^2}{4}.
$$
This will be clear from the next Theorem \ref{twotermasy}.
Note that the lower bound is independent of $\sigma$ and also on $R$ when $\sigma$ is large enough. We will in fact refine the estimate for hyperbolic balls by taking into account the value $R$ of the radius.

In the case $\sigma=+\infty$ (Dirichlet problem) it was proved in \cite{S1}, Theorem 5.6, that
\begin{equation}\label{infinity}
\frac{(n-1)^2}{4}+\frac{\pi^2}{R^2}-\frac{4\pi^2}{(n-1)R^3}\leq \lambda^D_1(B_R)\leq 
\frac{(n-1)^2}{4}+\frac{\pi^2}{R^2}+\dfrac{C}{R^3},
\end{equation}
where $C=\frac{\pi^2(n^2-1)}{2}\int_0^{\infty}\frac{r^2}{\sinh^2r}\,dr$.

\smallskip

For the Robin problem, and for $\sigma$ sufficiently large, we obtain the following calculation.

\begin{theorem} \label{twotermasy} Let $B_R$ be the ball of radius $R$ in the hyperbolic space ${\bf H}^n$ and let $\lambda_1(B_R,\sigma)$ be the first eigenvalue of the Robin Laplacian with parameter $\sigma>\frac{n-1}2$. 

\smallskip Then there are positive constants $R_0,c_0$ depending only on $\sigma$ and $n$ such that, for all $R\geq R_0$ one has:
$$
\frac{(n-1)^2}{4}+\frac{\pi^2}{R^2}-\frac{c_0}{R^3}\leq \lambda_1(B_R,\sigma)\leq 
\frac{(n-1)^2}{4}+\frac{\pi^2}{R^2}+\dfrac{C}{R^3},
$$
where $C=\frac{\pi^2(n^2-1)}{2}\int_0^{\infty}\frac{r^2}{\sinh^2r}\,dr$. (The upper bound holds for all $R$ and $\sigma$).

\smallskip

Consequently, for all $\sigma\in (\frac{n-1}{2},\infty]$ one has the following two-term  asymptotic expansion as $R\to\infty$:
$$
\lambda_1(B_R,\sigma)\sim \frac{(n-1)^2}{4}+\frac{\pi^2}{R^2}+O(\frac{1}{R^3}).
$$
\end{theorem}

We see that when the radius is large and the parameter $\sigma$ is greater than $\frac{n-1}2$, the eigenvalues $\lambda_1(B_R,\sigma)$ and $\lambda_1^D(B_R)$ are very very close;   in fact, the boundary parameter makes its appearance only in the remainder term,  which we find a bit surprising.  The constants $R_0$ and $c_0$ will be explicited in the proof. 

\smallskip

When $\sigma=+\infty$ (Dirichlet problem)  the lower bound in \eqref{infinity} was improved in \cite{Art} and the two-term expansion has been refined in Theorem 1.1 of \cite{Kri}. 

\smallskip

For simplicity we state the estimates in constant negative curvature $-1$. However the above estimates easily extend to arbitrary constant negative curvature, as follows: let $\Omega$ be any domain in ${\bf H}^n_{-\kappa^2}$, the hyperbolic space of constant curvature $-\kappa^2$ (we assume $\kappa>0$).  If $\sigma>0$ then
$$
\lambda_1(\Omega,\sigma)\geq\twosystem{\dfrac{(n-1)^2}{4}\kappa^2\quad\text{if}\quad \sigma\geq \dfrac{n-1}{2}\kappa}
{(n-1)\kappa\sigma-\sigma^2\quad\text{if}\quad 0<\sigma\leq\dfrac{n-1}{2}\kappa}
$$
If instead we assume $\sigma<0$, then
$
\lambda_1(\Omega,\sigma)\leq -\sigma^2+(n-1)\kappa\sigma.
$
Moreover, if $B_R$ denotes the ball of radius $R$ in ${\bf H}^n_{-\kappa^2}$, then one has a two-term asymptotic expansion as $R\to\infty$:
$$
\lambda_1(B_R,\sigma)\sim \frac{(n-1)^2}{4}\kappa^2+\frac{\pi^2}{R^2}+O(\frac{1}{R^3}).
$$
Observe that the second term is in fact independent on the ambient curvature $\kappa$ and on the boundary parameter $\sigma$. 

Theorems \ref{mckeanintro} and \ref{twotermasy} are proved in section \ref{fiveandsix}.


\section{Preliminary facts}\label{pre}

In this section we first prove some properties of eigenfunctions of the model one-dimensional problem, then we explain the method of proof of the comparison theorem, based on interior parallels and laplacian comparison. The exposition is based on \cite{S2}, but see also the Appendix in \cite{S1}. We chose to be here as self-contained as possible. 

\subsection{One-dimensional model problem} We will compare $\lambda_1(\Omega,\sigma)$ with the first eigenvalue $\lambda_1(R,\Theta, \sigma)$ of the following one-dimensional mixed problem on the interval $[0,R]$:
\begin{equation}\label{onedrobin}
\threesystem
{u''+\dfrac{\Theta'}{\Theta}u'+\lambda u=0}
{u'(0)=\sigma u(0)}
{u'(R)=0}
\end{equation}
with the weight function $\Theta(r)$ as in \eqref{theta}. Note that the boundary conditions are: Robin at $r=0$, Neumann at $r=R$; with these boundary conditions the spectrum of \eqref{onedrobin} is the spectrum of the operator 
$$
Lu\doteq -u''-\dfrac{\Theta'}{\Theta}u'
$$
acting on the weighted space $L^2([0,R],\mu)$ for the measure $\mu=\Theta (r)\,dr$.  Then, $L$ is self-adjoint and the spectrum is discrete:
$$
\lambda_1(R,\Theta,\sigma)\leq\lambda_2(R,\Theta,\sigma)\leq\dots\to+\infty.
$$
The min-max principle reads
\begin{equation}\label{minmaxoned}
\lambda_1(R,\Theta,\sigma)=\inf_{u\in H^1[0,R]}\Big\{\dfrac{\int_0^R u'(r)^2\Theta(r)\,dr+\sigma u(0)^2}{\int_0^Ru(r)^2\Theta(r)\,dr}\Big\}.
\end{equation}
 If $\sigma=0$ we have a Neumann weighted problem, the non-zero constants are eigenfunctions and then:
$$
\lambda_1(\Omega,\sigma)=\lambda_1(R,\Theta,0)=0.
$$
Therefore, we can assume $\sigma\ne 0$.

\begin{lemma}\label{oned}  Let  $u$ be a positive first eigenfunction of \eqref{onedrobin}.

\parte a If $\sigma>0$ then $u'>0$ on $[0,R)$.

\parte b If $\sigma<0$ then $u'<0$ on $[0,R)$.

\parte c If $\sigma >0$ and $\bar R<R$ then $\lambda_1(R,\Theta,\sigma)<\lambda_1(\bar R,\Theta,\sigma)$.

\medskip

In d) and e), we let $\sigma(r)=\frac{u'(r)}{u(r)}$ and assume that $\Theta$ is strictly log-concave (that is, $(\log\Theta)''< 0$ on $[0,R)$). 

\parte d If $\sigma>0$ then $0\leq \sigma(r)\leq\sigma$ for all $r\in [0,R]$. 

\parte e If $\sigma<0$ then $0\geq \sigma(r)\geq\sigma$ for all $r\in [0,R]$. 
\end{lemma}

\begin{proof}

\parte a Let $\sigma>0$ and let $\bar R$ be the first zero of $u'$. Note that $\bar R>0$ because $u'(0)=\sigma u(0)>0$.  We reason by contradiction and assume that $\bar R<R$. Then, by restriction, $u$ is also an eigenfunction of the same problem on $[0,\bar R]$; as $u$ is positive, it must be the first, hence:
$$
\lambda_1(\bar R,\Theta,\sigma)=\lambda_1(R,\Theta,\sigma).
$$
We show that this is impossible. 
Consider the function $w\in H^1[0,R]$ defined as follows:
\begin{equation}
w(r)=\twosystem
{u(r)\quad r\in[0,\bar  R]}
{u(\bar R)\quad r\in [\bar R,R].}
\end{equation}
We use $w$ as a test-function for the problem \eqref{onedrobin} on $[0,R]$ and therefore, by the min-max principle \eqref{minmaxoned}, we have :
\begin{equation}\label{odone}
\lambda_1(R,\Theta,\sigma)\int_0^Rw^2\Theta\leq\int_0^Rw'^2\Theta+\sigma w(0)^2
\end{equation}
 Now:
\begin{equation}\label{odtwo}
\int_0^Rw^2\Theta=\int_0^{\bar R}u^2\Theta+u(\bar R)^2\int_{\bar R}^{R}\Theta
>\int_0^{\bar R}u^2\Theta
\end{equation}
because $\Theta$ is positive on $[0,R)$ by assumption. 
On the other hand
\begin{equation}\label{odthree}
\begin{aligned}
\int_0^Rw'^2\Theta+\sigma w(0)^2&=\int_0^{\bar R}u'^2\Theta+\sigma u(0)^2\\
&=\int_0^{\bar R}uLu\cdot\Theta\\
&=\lambda_1(\bar R,\Theta,\sigma)\int_0^{\bar R}u^2\Theta
\end{aligned}
\end{equation}
Putting together \eqref{odone}, \eqref{odtwo} and \eqref{odthree} we get
$
\lambda_1(R,\Theta,\sigma)<\lambda_1(\bar R,\Theta,\sigma)
$
which is a contradiction as asserted. 

\parte b Let $\sigma<0$ and let $\bar R$ be the first zero of $u'$; note that $\bar R>0$. Assume $\bar R<R$. Let $v$ be the restriction of $u$ to $[\bar R,R]$. Then $v$ is a Neumann eigenfunction of $[\bar R,R]$ which implies
$
\lambda_1(R,\Theta,\sigma)\geq 0
$
because the Neumann spectrum is non-negative.
This is impossible because, as $\sigma<0$, we have $\lambda_1(R,\Theta,\sigma)< 0$ by the min-max principle (take $w=1$ as test-function).

\parte c One prolongs an eigenfunction of $[0,\bar R]$ on $[0,R]$ by a constant and uses the min-max principle. The construction is exactly as in part a).

\parte d Recall that $u$ satisfies 
$
u''+\frac{\Theta'}{\Theta}u'+\lambda u=0;
$
 an easy calculation shows that
$$
\sigma'=\Big(\dfrac{u'}{u}\Big)'=-\dfrac{\Theta'}{\Theta}\sigma-\lambda -\sigma^2.
$$
Now $\sigma$ is positive on $(0,R)$ by a); moreover $\sigma(0)=\sigma>0$ and $\sigma(R)=0$. It is then enough to show that $\sigma$ has no relative maximum in the open interval $(0,R)$. Assume by contradiction that $\bar r$ is one such. Then we would have
\begin{equation}\label{contra}
\sigma'(\bar r)=0, \quad \sigma''(\bar r)\leq 0.
\end{equation}
Now
$$
\sigma''=-\Big(\dfrac{\Theta'}{\Theta}\Big)'\sigma-\dfrac{\Theta'}{\Theta}\sigma'-2\sigma\sigma',
$$
hence:
$$
\sigma''(\bar r)=-(\log\Theta)''(\bar r)\sigma(\bar r)>0
$$
because $\sigma$ is strictly positive on $(0,R)$, which contradicts \eqref{contra}. Hence the assertion. 

\smallskip

\parte e We proceed in a similar way. We know that $u'<0$ on $[0,R)$, hence $\sigma(r)<0$ on that interval.  It is enough to show that $\sigma$ has no relative minimum in the open interval $(0,R)$. We assume  that $\bar r\in (0,R)$ is such a relative minimum and find a contradiction as before. 
\end{proof}


\subsection{Distance to the boundary and cut-locus} For complete details we refer to \cite{S2}. Let $\Omega$ be a compact domain with smooth boundary and let $\rho: \Omega\to\reals$ be the distance function to the boundary:
$$
\rho(x)={\rm dist}(x,\bd\Omega).
$$
The function $\rho$ is Lipschitz and, as $\bd\Omega$ is smooth, it is  smooth on a small tubular neighborhood of the boundary;  moreover, $\rho$ is singular precisely on the {\it cut-locus} ${\rm Cut}_{\bd\Omega}$, which is a closed set of measure zero in $\Omega$. Let us recall its definition. 

  Let $N_x$ be the inner unit normal at $x\in\bd\Omega$. Consider the unit speed geodesic starting at $x$ and going inside $\Omega$, in the direction normal to the boundary, that is, $\gamma_x(t)=\exp_x(tN_x)$. The {\it cut-radius} at $x$ is the positive number $c(x)$ defined as follows:

\nero {\it the geodesic $\gamma_x(t)$ minimizes distance to $\bd\Omega$ if and only if $t\in [0,c(x)]$.}

\medskip

Thus, we obtain the map $c:\bd\Omega\to [0,+\infty)$ which is known to be continuous; it is positive because $\bd\Omega$ is smooth (in fact, 
$\inf_{\bd\Omega}c$  is also called the {\it injectivity radius} of the normal exponential map). The cut-locus ${\rm Cut}_{\bd\Omega}$ is the closed subset of $\Omega$ defined by
$$
{\rm Cut}_{\bd\Omega}=\{\exp_x(c(x)N_x): x\in\bd\Omega\}.
$$
It is known that a point on the cut-locus is either a focal point along a normal geodesic, or is a point which can be joined to the boundary by at least two distinct minimizing geodesics. The cut locus has zero measure in $\Omega$; we let
$$
\Omega_{\rm reg}\doteq \Omega\setminus {\rm Cut}_{\bd\Omega},
$$
and call it the set of {\it regular points} of $\rho$. In fact, $\rho$ is $C^{\infty}$-smooth on $\Omega_{\rm reg}$ and there one has $\abs{\nabla\rho}=1$.
In conclusion, we have a disjoint union
$$
\Omega=\Omega_{\rm reg}\cup {\rm Cut}_{\bd\Omega}.
$$
Note that each point $p\in\Omega_{\rm reg}$ can be joined to the boundary by a unique geodesic segment minimizing distance. 


\subsection{Normal coordinates} We now consider the set:
$$
U\doteq\{(r,x)\in [0,\infty)\times\bd\Omega : 0\leq r<c(x)\}
$$
and see that the exponential map $\Phi:U\to\Omega_{\rm reg}$ defined by $\Phi(r,x)=\exp_x(rN_x)$ is actually a diffeomorphism. The pair $(r,x)$ gives rise to the {\it normal coordinates} of a regular point. We pull-back the Riemannian volume form by $\Phi$, and we write
$$
\Phi^{\star}dv_n(r,x)=\theta(r,x)\,dr\,dv_{n-1}(x),
$$
where the Jacobian $\theta(r,x)$ can then be seen as the density of the Riemannian measure in normal coordinates. Obviously $\theta$ is positive on $U$ and $\theta(0,x)=1$ for all $x\in\bd\Omega$. Any integrable function $f$ on $\Omega$ can be integrated in normal coordinates, as follows:
\begin{equation}\label{normal}
\int_{\Omega}f dv_n=\int_{\Omega_{\rm reg}}f dv_n=\int_{\bd\Omega}\int_0^{c(x)}f(r,x)\theta(r,x)dr\, dv_{n-1}(x).
\end{equation}
where we identify a regular point of $\Omega$ with its normal coordinates. 
The map $r\mapsto \theta(r,x)$ is smooth and  extends by continuity on $[0,c(x)]$ :
$$
\theta(c(x),x)=\lim_{r\to c(x)_-}\theta(r,x).
$$
Note that $\theta(c(x),x)$ could be zero; it is zero precisely at the focal points of the boundary. 
The function 
$$
\Delta_{\rm reg}\rho\doteq\Delta(\rho\vert_{\Omega_{\rm reg}})
$$
is the Laplacian of $\rho$ restricted to the regular points. It is then smooth on $\Omega_{\rm reg}$, and it is also in $L^1(\Omega)$ (see \cite{S2}). In normal coordinates it has the following expression:
\begin{equation}\label{deltareg}
\Delta_{\rm reg}\rho(r,x)=-\dfrac{\theta'(r,x)}{\theta(r,x)}
\end{equation}
where for simplicity $\theta'$ refers to differentiation with respect to $r$. For a proof see \cite{G} p. 40. Note also that $\Delta_{\rm reg}\rho(r,x)$ is $(n-1)$-times the mean curvature at $(r,x)$ of the level set $\rho=r$. 
By the classical Heintze-Karcher volume estimates we see that all regular points $(r,x)$ one has the inequality:
\begin{equation}\label{hk}
\Delta_{\rm reg}\rho(r,x)\geq -\dfrac{\Theta'(r)}{\Theta(r)}
\end{equation} 
where $\Theta$ has been defined in \eqref{theta}.


\subsection{Distributional Laplacian and main technical lemma} As $\rho$ is only Lipschitz, we will define its Laplacian in the distributional sense,  as the pairing
\begin{equation}\label{pairing}
\Big(\Delta\rho,\phi\Big)\doteq \int_{\Omega}\rho\Delta\phi
\end{equation}
for all smooth functions $\phi$ which are compactly supported in the interior of $\Omega$. We have the following lemma.

\begin{lemme}\label{prep} \parte a The distribution $\Delta\rho$ splits:
\begin{equation}\label{splits}
\Delta\rho=\Delta_{\rm reg}\rho+\Delta_{\rm cut}\rho
\end{equation}
where $\Delta_{\rm reg}\rho\in L^1(\Omega)$ is as in \eqref{deltareg}, and where $\Delta_{\rm cut}\rho$ is a distribution supported on the cut-locus ${\rm Cut}_{\bd\Omega}$ and defined by:
\begin{equation}\label{deltacut}
\Big(\Delta_{\rm cut}\rho,\phi\Big)=\int_{\bd\Omega}\phi(c(x),x)\theta(c(x),x)\, dv_{n-1}(x).
\end{equation}

\parte b  $\Delta_{\rm cut}\rho$ is a positive distribution of order zero, hence a (positive) Radon measure; thus the pairing \eqref{pairing} can be extended to any continuous function $\phi$ on $\Omega$.  

\parte c As a consequence of the splitting \eqref{splits}, the positivity of $\Delta_{\rm cut}\rho$  and \eqref{hk} we have, in the distributional sense:
$$
\Delta\rho\geq  -\dfrac{\Theta'}{\Theta}\circ\rho,
$$
where $\Theta(r)$ is our weight function, depending on the curvature data $(K,H)$ and defined in \eqref{theta}.
\end{lemme}

\begin{proof} We first observe that we have, for all $\phi\in C^{\infty}_c(\Omega)$:
$$
\Big(\Delta\rho,\phi\Big)=\int_{\Omega}\rho\Delta\phi\, dv_n=\int_{\Omega}\scal{\nabla\rho}{\nabla\phi}\, dv_n,
$$
because $\rho$ is Lipschitz hence also in $H^1(\Omega)$. Integrating in normal coordinates we see:
$$
\int_{\Omega}\scal{\nabla\rho}{\nabla\phi}dv_n=\int_{\bd\Omega}\int_0^{c(x)}\phi'(r,x)\theta(r,x)\,dr\,dv_{n-1}(x)
$$
Integrating by parts, since $\phi(0,x)=0$:
$$
\int_0^{c(x)}\phi'(r,x)\theta(r,x)\,dr=\phi(c(x),x)\theta(c(x),x)+\int_0^{c(x)}\phi(r,x)\Big(-\dfrac{\theta'(r,x)}{\theta(r,x)}\Big)\theta(r,x)\,dr.
$$
Integrating on $\bd\Omega$ and taking into account \eqref{deltareg} and \eqref{deltacut} we see:
$$
\int_{\Omega}\scal{\nabla\rho}{\nabla\phi}dv_n=\Big(\Delta_{\rm cut}\rho,\phi\Big)+\int_{\Omega}\phi\Delta_{\rm reg}\rho\,dv_n,
$$
which shows the splitting \eqref{splits}.
\end{proof}
As a matter of notation, from now on we will write:
$$
\Big(\Delta\rho,\phi\Big)=\int_{\Omega}\phi\Delta\rho,
$$
where on the right it is understood the integral of $\phi$ with respect to the measure $\Delta\rho=\Delta_{\rm reg}\rho+\Delta_{\rm cut}\rho$. Note that $\phi$ can be any continuous function, not necessarily supported in the interior of $\Omega$.


\subsection{Main lemma} 
The lemma which follows is proved by integrating in normal coordinates, as we did in Lemma \ref{prep}. 

\begin{lemme}\label{tech}
 Let $u:[0,R]\to\reals$ be smooth and consider the function $v=u\circ\rho$ on $\Omega$. Then:

\parte a  $v$ is Lipschitz. 

\parte b One has, in the sense of distributions:
$$
\Delta v=\Delta(u\circ\rho)=-u''\circ\rho+(u'\circ\rho)\Delta\rho.
$$

\parte c Green's formula holds: if $v=u\circ\rho$ and $f\in C^2(\Omega)$ then
$$
\int_{\Omega}\Big(f\Delta v-v\Delta f\Big)=\int_{\bd\Omega}\Big(f\derive vN-v\derive fN\Big)dv_{n-1},
$$
where $\Delta v$ is taken in the sense of distributions, as in b).
\end{lemme}

About the proof: first observe that a) is immediate. For b), take a test-function $\phi$ and observe that 
$$
\Big(\Delta (u\circ\rho),\phi\Big)=\int_{\Omega}(u\circ\rho)\Delta\phi=
\int_{\Omega}(u'\circ\rho)\scal{\nabla\rho}{\nabla\phi}dv_n,
$$
where the last equality holds because $\rho$ is Lipschitz (hence in $H^1(\Omega)$), and $\nabla(u\circ\rho)=(u'\circ\rho) \nabla\rho$. We integrate the last term in normal coordinates, and then by parts in the inner integral and the equality follows. Finally, for c), we compute  $\int_{\Omega}\scal{\nabla f}{\nabla v}dv_n$ in two different ways. First:
\begin{equation}\label{firstway}
\int_{\Omega}\scal{\nabla f}{\nabla v}dv_n=\int_{\Omega}v\Delta f-\int_{\bd\Omega}v\derive fN,
\end{equation}
which is easily shown to hold because $v$ is Lipschitz and $f$ is smooth. On the other hand:
\begin{equation}\label{otoh}
\begin{aligned}
\int_{\Omega}\scal{\nabla f}{\nabla v}dv_n&=\int_{\Omega_{\rm reg}}\scal{\nabla f}{\nabla v}dv_n\\
&=\int_{\Omega_{\rm reg}}(u'\circ\rho)\scal{\nabla f}{\nabla\rho}dv_n\\
&=\int_{\bd\Omega}\int_0^{c(x)}\Big(u'(r)f'(r,x)\theta(r,x)\,dr\Big)\,dx
\end{aligned}
\end{equation}
Now we integrate by parts in the inner integral and equate the two expressions; after some work we get the desired identity. We omit further details because they are straightforward.


\section{Proof of the comparison theorem} \label{scomparison}

\subsection{Proof of the lower bound in Theorem \ref{compintro}} \label{prooflb} We assume $\sigma>0$ and fix a positive first eigenfunction $u$ of our one-dimensional model problem on $[0,R]$, associated to $\bar\lambda$. It satisfies: 
$$
\threesystem
{u''+\dfrac{\Theta'}{\Theta}u'+\bar\lambda u=0}
{u'(R)=0}
{u'(0)=\sigma u(0)}
$$
Consider the pull-back function on $\Omega$ given by
$
v=u\circ\rho.
$
By Lemma \ref{oned} we know that $u'\geq 0$ on $[0,R]$. Hence, by Lemma \ref{tech}b and Lemma \ref{prep}c
$$
\begin{aligned}
\Delta v&=-u''\circ\rho+(u'\circ\rho)\Delta\rho\\
&\geq \Big(-u''-\dfrac{\Theta'}{\Theta}u'\Big)\circ\rho\\
&=\bar\lambda (u\circ\rho),
\end{aligned}
$$
that is, 
\begin{equation}\label{deltaine}
\Delta v\geq \bar\lambda v.
\end{equation}
Next, we consider a first positive eigenfunction $f$ of our 
Robin problem
\begin{equation}\label{robintext}
\twosystem
{\Delta f=\lambda f\quad\text{on}\quad \Omega,}
{\derive fN=\sigma f\quad\text{on}\quad \bd\Omega.}
\end{equation}
We multiply \eqref{deltaine} by $f$, the first equation of \eqref{robintext} by $v$ and subtract. We obtain
$$
f\Delta v-v\Delta f\geq(\bar\lambda-\lambda)fv.
$$
Now, by Lemma \ref{tech}c 
$$
\int_{\Omega}\Big(f\Delta v-v\Delta f\Big)=\int_{\bd\Omega}\Big(f\derive vN-v\derive fN\Big)dv_{n-1}=0,
$$
simply because, on $\bd\Omega$:
$$
\derive vN=u'(0)=\sigma u(0)=\sigma v\quad\text{and}\quad \derive fN=\sigma f.
$$
We conclude that
$$
0\geq (\bar\lambda-\lambda)\int_{\Omega}fv.
$$
As $f$ and $v$ are both positive we must have $\bar\lambda-\lambda\leq 0$ as asserted.


\subsection{Proof of the upper bound of Theorem \ref{compintro}}
Now assume $\sigma<0$ and define $u$ and $v$ as in the previous case. Lemma \ref{oned} now says that $u'\leq 0$, hence we see
$
\Delta v\leq\bar\lambda v.
$
The proof goes exactly as before, with the inequalities reversed. The assertion follows.


\subsection{Proof of Theorem \ref{zerocurvature}} Recall that we have to show that if $K=H=0$ then 
$$
\lambda_1(\Omega,\sigma)\geq \lambda_1([0,2R],\sigma),
$$
and that if $\Omega$ is a flat cylinder then equality holds. 

\smallskip

\begin{proof} We take $K=H=0$ hence $\Theta(r)=1$ for all $r$. Then, $\lambda_1(R,\Theta,\sigma)$ is the first eigenvalue of the problem:
\begin{equation}\label{problemone}
\threesystem
{u''+\lambda u=0}
{u'(0)=\sigma u(0)}
{u'(R)=0}
\end{equation}
On the other hand, $\lambda_1([0,2R],\sigma)$ is the first eigenvalue of the following problem on $[0,2R]$:
\begin{equation}\label{problemtwo}
\threesystem
{u''+\lambda u=0}
{u'(0)=\sigma u(0)}
{u'(2R)=-\sigma u(2R)}
\end{equation}
We show that problems \eqref{problemone} and \eqref{problemtwo} have in fact the same the first eigenvalue. 
First we observe that \eqref{problemtwo} is invariant under the symmetry $r\to 2R-r$. This means that, if we fix a positive first eigenfunction $u$
of problem \eqref{problemtwo},  then $u$ must be either even or odd at $r=R$:  as $u$ is positive it must be even at $R$, so that $u'(R)=0$. Hence $u$ is also an eigenfunction of problem \eqref{problemone}, necessarily the first (again, because it is positive).  In conclusion:
$$
\lambda_1(R,\Theta,\sigma)=\lambda_1([0,2R],\sigma),
$$
as asserted. 

\smallskip

Now assume that $\Omega=[0,2R]\times\Sigma$ with the product metric. Then we can separate variables and see that $L^2(\Omega)$ admits a basis of eigenfunctions of the Robin problem of type $f(r,x)=u(r)\phi(x)$ where $r\in [0,R]$ and $x\in\Sigma$. Here $u$ is radial and satisfies \eqref{problemtwo}, and $\phi$ is an eigenfunction of the Laplacian on the closed manifold $\Sigma$. Clearly an eigenfunction associated to $\lambda_1(\Omega,\sigma)$ must correspond to the case where  $\phi(x)$ is a (non-zero) constant. Hence $f(r,x)=cu(r)$ is actually radial. This shows that equality holds for any flat cylinder, and the proof is complete.  
\end{proof}


\subsection{Proof of Corollary \ref{convex}} \label{proofconvex} It is enough to show that
the first eigenvalue of \eqref{problemone} satisfies, for $\sigma>0$:
$$
\lambda_1>
\dfrac{\pi^2\sigma}{\pi^2R+4R^2\sigma},
$$
while, if $\sigma<0$, one has $\lambda_1<-\sigma$.

\begin{proof} Case $\sigma>0$. The spectrum is positive and any eigenfunction of problem \eqref{problemone} has the form
$$
u(r)=a\sin(\sqrt{\lambda}r)+b\cos(\sqrt{\lambda}r)
$$
for $\lambda>0$. An easy calculation shows that the boundary conditions force $a\ne 0$ and
$$
\frac{b}{a}=\frac{\sqrt{\lambda}}{\sigma}=\cot(R\sqrt{\lambda}).
$$
If one sets $x=R\sqrt{\lambda}$ then $x\tan x=R\sigma$ hence $x$ must be a positive zero of the function $\phi(x)=x\tan x-R\sigma$. In conclusion, we see that the first eigenvalue of \eqref{problemtwo} is given by :
$$
\lambda_1=\dfrac{c_1}{R^2},
$$
where $c_1$ is the first the positive zero of $\phi(x)=x\tan x-R\sigma$. Now  observe that $\phi(0)$ is negative and $\phi(x)$ gets large positive when $x$ is close to $\frac{\pi}2$. This means that $c_1\in (0,\frac{\pi}{2})$; by the Becker-Starck inequality \cite{BS}:
$$
\tan x<\frac{\pi^2 x}{\pi^2-4x^2}, \quad \text{for all}\quad x\in (0,\frac{\pi}{2}),
$$
we see:
$$
R\sigma=c_1\tan c_1<\frac{\pi^2 c_1^2}{\pi^2-4c_1^2}\quad\text{hence}\quad c_1^2>\dfrac{\pi^2R\sigma}{\pi^2+4R\sigma},
$$
hence
$$
\lambda_1=\dfrac{c_1^2}{R^2}>\dfrac{\pi^2 \sigma}{\pi^2R+4R^2\sigma},
$$
which is the desired inequality. 

\medskip

Case $\sigma<0$. In that case $\lambda\doteq\lambda_1<0$ and an associated eigenfunction is of type:
$$
u(r)=a\sinh(\sqrt{\abs{\lambda}}r)+b\cosh(\sqrt{\abs{\lambda}}r).
$$
The boundary conditions give
$$
\frac{b}{a}=\frac{\sqrt{\abs{\lambda}}}{\sigma}=-\coth(R\sqrt{\abs{\lambda}}),
$$
so that, if $x=R\sqrt{\abs{\lambda}}$, then $R\sigma\coth x+x=0$, which means that 
$R\sqrt{\abs{\lambda_1}}=c_1$, where $c_1$ is the unique positive root of $\phi(x)\doteq R\sigma\coth x+x$. Now:
$$
0=\phi(c_1)=R\sigma\coth c_1+c_1< R\sigma+c_1,
$$
hence $c_1>-R\sigma$; consequently $\sqrt{\abs{\lambda_1}}>-\sigma=\abs{\sigma}$ and squaring both sides we get the assertion. 

\end{proof}


%



\section{Proof of domain monotonicity}\label{monotonicity}

In this section we prove Theorem \ref{dmintro}.

\subsection{The first Robin eigenvalue of a revolution manifold} Let $B(x_0,R)$ be a geodesic ball centered at the pole $x_0$ of a revolution manifold.
It is a standard fact that, by rotational invariance,  the first eigenfunction of the Robin Laplacian on $B(x_0,R)$ is radial:  $v=v(r)$ where $r$ is the distance to the pole; consequently,  the first eigenvalue of $B(x_0,R)$ with parameter $\sigma$ is the first eigenvalue of the following one-dimensional problem:
\begin{equation}\label{onedtwo}
\threesystem
{v''+(n-1)\dfrac{\Phi'}{\Phi}v'+\lambda v=0}
{v'(0)=0}
{v'(R)=-\sigma v(R)}
\end{equation}
Note that the condition $v'(0)=0$ is imposed to have regularity at $x_0$.
For example, for the geodesic ball in hyperbolic space ${\bf H}^n$ we have $\Phi(r)=\sinh r$ hence the problem becomes:
\begin{equation}\label{hypballs}
\threesystem
{v''+(n-1)(\coth r)v'+\lambda v=0}
{v'(0)=0}
{v'(R)=-\sigma v(R).}
\end{equation}

It will be convenient to parametrize instead by the distance $\rho$ to the boundary of $B(x_0,R)$. As $\rho=R-r$ we set:
$$
u(r)=v(R-r), \quad \Theta(r)=\Phi(R-r)^{n-1}.
$$
Note that $\Theta$ is positive on $[0,R)$; a calculation shows:
$$
\dfrac{\Theta'(r)}{\Theta(r)}=-(n-1)\dfrac{\Phi'(R-r)}{\Phi(R-r)},\quad 
(\log\Theta)''(r)=(n-1)(\log\Phi)''(R-r).
$$
Therefore,  $\Phi$ is log-concave if and only if  $\Theta$ is log-concave. Problem \eqref{onedtwo} becomes the equivalent problem:
\begin{equation}\label{onedthree}
\threesystem
{u''+\dfrac{\Theta'}{\Theta}u'+\lambda u=0}
{u'(0)=\sigma u(0)}
{u'(R)=0.}
\end{equation}


\subsection{Domain monotonicity} We now prove Theorem \ref{dmintro}:

\begin{theorem}\label{dm} Let $\Omega$ be a domain of a revolution manifold $M$ with pole $x_0$, whose warping function $\Phi$ is log-concave.  Assume that $\Omega\subseteq B(x_0,R)$. If $\sigma>0$ then
$$
\lambda_1(\Omega,\sigma)\geq \lambda_1(B(x_0,R),\sigma),
$$
while if $\sigma<0$ then the opposite inequality holds:
$
\lambda_1(\Omega,\sigma)\leq \lambda_1(B(x_0,R),\sigma),
$
\end{theorem}

Note that the theorem applies in any space-form $M_K$.

\begin{proof} Assume first $\sigma>0$ and set for brevity $B\doteq B(x_0,R)$ and $\lambda=\lambda_1(B(x_0,R),\sigma)$. The first positive eigenfunction $\phi$ of $B$ is radial, and depends only on the distance to the boundary of $B$, hence it is written $\phi=u\circ\rho$ where $\rho$ is the distance function to the boundary. Therefore:
$$
 \nabla\phi=(u'\circ\rho)\nabla\rho
$$
where $u$ solves \eqref{onedthree}. 
Define a function $\sigma^{\star}:\bd\Omega\to\reals$ by the rule:
\begin{equation}\label{sigmastar}
\sigma^{\star}(x)=\dfrac{1}{\phi(x)}\derive{\phi}{N}(x),
\end{equation}
where $N$ is the inner unit normal to $\bd\Omega$. Set $\rho(x)=r$ and observe that, by Lemma \ref{oned}, $u'\geq 0$. Then:
$$
\derive{\phi}{N}(x)=\scal{\nabla\phi(x)}{N(x)}=u'(r)\scal{\nabla\rho (x)}{N(x)}\leq u'(r),
$$
because $\nabla\rho$ is of unit length so that $\scal{\nabla\rho}{N}\leq 1$.
This means that
$$
\dfrac{1}{\phi(x)}\derive{\phi}{N}(x)\leq \dfrac{u'(r)}{u(r)}.
$$
Since $\Theta$ is log-concave, again by Lemma \ref{oned} we see:
$$
\dfrac{u'(r)}{u(r)}\leq\sigma \quad\text{hence also}\quad \sigma^{\star}(x)\leq \sigma,
$$
for all $x\in\bd\Omega$. Now the restriction of $\phi$ to $\Omega$ which, by a slight abuse of language, we keep denoting by $\phi$, satisfies:
$$
\twosystem
{\Delta\phi=\lambda \phi\quad\text{on}\quad \Omega}
{\derive{\phi}{N}=\sigma^{\star}\phi\quad\text{on}\quad \bd\Omega.}
$$
Since $\phi$ is positive on $\Omega$, it is the first eigenfunction of that problem, hence:
$$
\lambda\doteq\lambda_1(B(x_0,R),\sigma)=\lambda_1(\Omega,\sigma^{\star}).
$$
But now, as $\sigma^{\star}(x)\leq\sigma$ we immediately see from monotonicity that
$
\lambda_1(\Omega,\sigma^{\star})\leq \lambda_1(\Omega,\sigma),
$
and the assertion follows. 

\medskip

Now assume $\sigma<0$. The proof in that case is similar, with all signs reversed; in particular $u'\leq 0$, $\sigma(r)\leq 0$ etc. One defines $\sigma^{\star}(x)$ as before, and since $\scal{\nabla\rho}{N}\leq 1$ we see
$$
\derive{\phi}{N}(x)\geq u'(r)
$$
hence
$
\sigma^{\star}(x)\geq \sigma.
$
By considering the restriction of $\phi$ to $\Omega$ as before  we conclude from monotonicity that
$
\lambda_1(B(x_0,r),\sigma^{\star})=\lambda_1(\Omega,\sigma^{\star})\geq \lambda_1(\Omega,\sigma),
$
and the assertion follows. 

\end{proof}


\section{Proof of McKean-type inequality}\label{mckean} In this section we prove Theorem \ref{mckeanintro} and Theorem \ref{twotermasy}. The idea is simply to work with the ODE \eqref{hypballs} and approximate $\coth r$ by $1$.


\subsection{A preparatory lemma} 

\begin{lemme} \label{preplemma} Given a positive constant $A$, consider the following mixed Robin problem on $[0,R]$:
 \begin{equation}\label{robinsimpletwo}
\threesystem
{u''+2Au'+\lambda u=0}
{u'(0)=0}
{u'(R)=-\sigma u(R)}
\end{equation}

\parte a One has, for all $R>0$:
$$
\lambda_1\geq \twosystem
{2\sigma A-\sigma^2\quad\text{if}\quad 0\leq\sigma\leq A}
{A^2\quad\text{if}\quad \sigma\geq A}
$$

\parte b Now assume $\sigma<0$. Then, for all $R>0$:
$$
\lambda_1\leq -\sigma^2+2A\sigma.
$$

\parte c Assume $\sigma>A$. There are positive constants $R_0,c_0$ depending only on $A$ and $\sigma$ such that, for all $R\geq R_0$ one has:
$$
\lambda_1\geq A^2+\frac{\pi^2}{R^2}-\frac{c_0}{R^3}.
$$
\end{lemme}

\begin{proof} We set, for brevity, $\lambda\doteq\lambda_1$ and assume, at first, that  $\lambda<A^2$. Then, the solutions of the ODE:
\begin{equation}\label{ODE}
\twosystem
{u''+2Au'+\lambda u=0}
{u'(0)=0}
\end{equation}
are all multiples of the function:
\begin{equation}\label{sinh}
u(x)=e^{-Ax}\Big(A\sinh(qx)+q\cosh(qx)\Big)
\end{equation}
where
$
q=\sqrt{A^2-\lambda}.
$
One computes
$
u'(x)=-\lambda e^{-Ax}\sinh(qx),
$
and the boundary condition $u'(R)=-\sigma u(R)$ gives:
\begin{equation}\label{exact}
\lambda=\sigma(A+q\coth(qR)).
\end{equation}
\parte a Assume $\sigma\geq A$. As $\coth(qR)>1$ one has 
$
\lambda>\sigma A.
$
By hypothesis we have also $\lambda<A^2$, hence  we see that
$
A^2>\sigma A
$
hence $\sigma<A$: contradiction. 

The conclusion is that
 if $\sigma\geq A$ then $\lambda\geq A^2$.

\smallskip

Now assume $\sigma\leq A$. Then either $\lambda\geq A^2$ (and then, a fortiori, $\lambda\geq 2\sigma A-\sigma^2$, and we are done) or  $\lambda< A^2$ and the eigenfunction is as in \eqref{sinh}. In that case, equation \eqref{exact} implies
$
\lambda\geq \sigma A+\sigma q,
$
that is:
$$
\lambda-\sigma A\geq \sigma\sqrt{A^2-\lambda}.
$$
Squaring both sides we see
$$
\lambda\geq 2\sigma A-\sigma^2.
$$
In both cases we see that, if $\sigma\leq A$ then
$
\lambda\geq 2\sigma A-\sigma^2,
$
as asserted.

\smallskip

\parte b If $\sigma<0$  then $\lambda<0$; the exact expression \eqref{exact} can be written:
$$
\abs{\lambda}=\abs{\sigma} A+\abs{\sigma}q\coth(qR).
$$
Hence 
$
\abs{\lambda}-\abs{\sigma}A\geq\abs{\sigma}q.
$
Squaring both sides and proceeding as before we arrive at the upper bound
$
\lambda\leq -\sigma^2+2\sigma A.
$
\smallskip

It remains to show c). By assumption $\sigma>A$ hence $\lambda>A^2$ by a). The solutions of \eqref{ODE} are now multiples of:
$$
u(x)=e^{-Ax}\Big(A\sin(qx)+q\cos(qx)\Big), \quad\text{where}\quad q=\sqrt{\lambda-A^2}.
$$
Then $u'(x)=-\lambda e^{-Ax}\sin(qx)$ and the boundary condition $u'(R)=-\sigma u(R)$ gives 
$
\lambda=\sigma A+\sigma q\cot(qR).
$
Hence 
$$
\lambda-A^2=\sigma q\cot(qR)+\sigma A-A^2;
$$
multiplying by $R^2$:
$$
q^2R^2=(\sigma R) (qR)\cot(qR)+(\sigma A-A^2)R^2.
$$
Setting $t=qR$ we see:
$$
t^2-\sigma R t\cot t-(\sigma A-A^2)R^2=0.
$$
Conclude that  the eigenvalues are given by
$$
\lambda_k=A^2+\dfrac{x_k^2}{R^2}
$$
where $\{x_1, x_2,\dots\}$ is the sequence of positive zeroes of the function
$$
\phi(x)=x^2-\alpha x\cot x-\beta
$$
where  $\alpha=\sigma R$ and $\beta=(\sigma A-A^2)R^2$.

\smallskip

We need to estimate $x_1$. Set:
$$
R_0\doteq\max\{\dfrac{2\pi}{\sqrt{\sigma A-A^2}}, \dfrac{4\sigma}{3(\sigma A-A^2)}\}.
$$
If $x_1\geq \pi$ inequality c) follows immediately. Then, in what follows we will assume $x_1<\pi$. We first want to show that if $R\geq R_0$ then $x_1>\frac{\pi}2$. In fact, in that case 
$
\beta\doteq (\sigma A-A^2)R^2\geq 4\pi^2
$
hence, as $x_1^2=\alpha x_1\cot x_1+\beta$ (by definition), one has
$$
x_1^2\geq \alpha x_1\cot x_1+4\pi^2.
$$
If $\cot x_1\geq 0$ the inequality gives $x_1> 2\pi$ and c) follows; then $\cot x_1<0$ so that   $x_1> \frac{\pi}{2}$ and, by our initial assumption
$$
x_1\in \Big(\frac{\pi}2,\pi\Big).
$$
The definition of $x_1$ gives
$
\cot x_1=\dfrac{x_1^2-\beta}{\alpha x_1},
$
or, equivalently,
$
\tan x_1=\dfrac{\alpha x_1}{x_1^2-\beta}.
$
As $\pi-x_1\in (0,\frac{\pi}2)$ and $\tan (\pi-x_1)=-\tan x_1$ it holds:
$$
\tan (\pi-x_1)=\dfrac{\alpha x_1}{\beta-x_1^2}.
$$
Considering that $\tan (\pi-x_1)\geq \pi-x_1$ we arrive at:
\begin{equation}\label{pixone}
\pi-x_1\leq \dfrac{\alpha x_1}{\beta-x_1^2}.
\end{equation}
Since $x_1\leq  \pi$:
\begin{equation}\label{pixtwo}
\dfrac{\alpha x_1}{\beta-x_1^2}=\dfrac{\sigma Rx_1}{(\sigma A-A^2)R^2-x_1^2}\leq
\dfrac{\pi\sigma R}{(\sigma A-A^2)R^2-\pi^2}.
\end{equation}

If $R\geq R_0$ then $R\geq \frac{2\pi}{\sqrt{\sigma A-A^2}}$ and one checks that
\begin{equation}\label{pixthree}
\dfrac{\pi\sigma R}{(\sigma A-A^2)R^2-\pi^2}\leq \dfrac{4\pi\sigma}{3(\sigma A-A^2)}\cdot\dfrac{1}{R}.
\end{equation}
By \eqref{pixone}, \eqref{pixtwo} and \eqref{pixthree}:
$$
x_1\geq \pi-\dfrac{4\pi\sigma}{3(\sigma A-A^2)}\cdot\dfrac 1R,
$$
If $R\geq R_0$ the right hand side is non-negative, and squaring both sides we see that
$$
x_1^2\geq \pi^2-\dfrac{c_0}{R}, \quad\text{with}\quad c_0=\dfrac{8\pi\sigma}{3(\sigma A-A^2)}.
$$
Eventually, when $R\geq R_0$,  we obtain:
$$
\dfrac{x_1^2}{R^2}\geq \dfrac{\pi^2}{R^2}-\dfrac{c_0}{R^3},
$$
with
$$
R_0\doteq\max\Big\{\dfrac{2\pi}{\sqrt{\sigma A-A^2}}, \dfrac{4\sigma}{3(\sigma A-A^2)}\Big\}, \quad\text{and}\quad c_0=\dfrac{8\pi\sigma}{3(\sigma A-A^2)}.
$$
The proof is complete.

\end{proof}


\subsection{Proof of Theorems \ref{mckeanintro} and \ref{twotermasy}}\label{fiveandsix}

Let $\Omega$ be a (bounded) domain in ${\bf H}^n$. We first assume $\sigma>0$. Now $\Omega\subseteq B(x_0,R)$ for a suitable ball; by the monotonicity proved in Theorem \ref{dm} we see that, for $\sigma>0$:
$$
\lambda_1(\Omega,\sigma)\geq \lambda_1(B(x_0,R),\sigma).
$$
 Therefore, we proceed to estimate the first Robin eigenvalue of hyperbolic balls of radius $R$, which is the following problem on $[0,R]$ (see \eqref{hypballs}):
\begin{equation}\label{hypball}
\threesystem
{v''+(n-1)(\coth r)v'+\lambda v=0}
{v'(0)=0}
{v'(R)=-\sigma v(R).}
\end{equation}
As remarked before, the change $u(r)=v(R-r)$ transforms \eqref{hypballs} in the problem \eqref{onedthree} of the type considered in Lemma \ref{oned}: one then has $u'\geq 0$ on $[0,R]$ hence $v'\leq 0$ on $[0,R]$. As $\coth r\geq 1$ we see that 
$$
v''+(n-1)v'+\lambda v\geq 0.
$$
Note that if $\sigma<0$ the inequality is reversed. The conclusion is that 

\begin{lemme} \label{linear} If $\sigma>0$ (resp. $\sigma<0$)  then $\lambda_1(B(x_0,R),\sigma)$ is larger than or equal to (resp. less than or equal to) the first eigenvalue of the problem:
\begin{equation}\label{linearproblem}
\threesystem
{v''+(n-1)v'+\lambda v=0}
{v'(0)=0}
{v'(R)=-\sigma v(R).}
\end{equation}
\end{lemme}


\subsection{Proof of Theorem \ref{mckeanintro}} Given Lemma \ref{linear}, we apply Lemma \ref{preplemma} for $A=\frac{n-1}2$. If $\sigma>0$ we get immediately:
$$
\lambda_1(\Omega,\sigma)\geq\twosystem{\dfrac{(n-1)^2}{4}\quad\text{if}\quad \sigma\geq \dfrac{n-1}{2}}
{(n-1)\sigma-\sigma^2\quad\text{if}\quad 0<\sigma\leq\dfrac{n-1}{2}}
$$
If instead we assume $\sigma<0$, then
$
\lambda_1(\Omega,\sigma)\leq -\sigma^2+(n-1)\sigma.
$


\subsection{Proof of Theorem \ref{twotermasy}} Again, we  apply Lemma \ref{preplemma}c for $A=\frac{n-1}{2}$. We obtain:
$$
\lambda_1(B_R,\sigma)\geq \frac{(n-1)^2}{4}+\frac{\pi^2}{R^2}-\frac{c_0}{R^3}
$$
for $R\geq R_0$. The upper bound in Theorem \ref{twotermasy} follows because 
$\lambda_1(B_R,\sigma)\leq\lambda_1^D(B_R)$ and the upper bound in \eqref{infinity}. 


\section{Appendix : model domains} \label{appendix}


We wish to construct, for any choice of $K,H$ and $R$, an n-dimensional domain $\bar\Omega\doteq\bar\Omega(K,H,R)$ with boundary components $\Gamma_1$ and $\Gamma_2$ such that the first eigenvalue of $\bar\Omega$ with Robin conditions on $\Gamma_1$ and Neumann conditions on $\Gamma_2$ coincides with $\lambda_1(R,\Theta,\sigma)$.

\medskip

{\bf Case 1.} It covers three distinct situations: 
a) $K>0$ and $H\in\reals$, 
b) $K=0$ and  $H\ne 0$, 
c) $K<0$ and  $\abs{H}>\sqrt{\abs K}$. 

\smallskip

In all these cases $\bar\Omega$ will be an annulus in $M_K$, the simply connected manifold with constant curvature $K$. 
Recall that the  $M_K$ is a revolution manifold with metric 
$
g=dr^2+s_K(r)^2g_{\sphere{n-1}}
$
and that the coordinate $r$ is geodesic distance to the pole $\{O\}$ of $M$. The mean curvature of the ball with center the pole and radius $r$, with respect to the inner unit normal $N=-\nabla r$ is
$$
H(r)=\cot_K(r)\doteq \dfrac{s'_K(r)}{s_K(r)}
$$
Let us set 
$
A=\cot^{-1}_K(\abs{H})
$
which is well-defined given our conditions on $H$. We remark that, if $\Omega$ has curvature data $(K,H)$, with $H>0$, then its inner radius $R\leq A$ (see \cite{S2}).  Our model domain will be the  annulus defined as follows:
$$
\bar\Omega=\bar\Omega(K,H,R)=
\twosystem
{A\leq r\leq A+R\quad\text{if $H<0$}}
{A-R\leq r\leq A\quad\text{if $H\geq 0$}}
$$
Note that the boundary of $\bar\Omega$ consists of two pieces; we call $\Gamma_1$ the component where $r=A$ and $\Gamma_2$ the other. One checks that the mean curvature of $\bd\bar\Omega$ is constant, equal to $H$, on $\Gamma_1$. 

\medskip

{\bf Case 2.} $H=0, K=0$. Then, we simply take the flat cylinder $\Omega=[0,R]\times \sphere{n-1}$, and let $\Gamma_1=\{0\}\times\sphere{n-1}$.

\medskip

{\bf Case 3.} $K<0, H\in(-\sqrt{\abs K}, \sqrt{\abs K})$. As ambient manifold we take the hyperbolic cylinder $\tilde M_K$, which is the rotationally invariant manifold $(-\infty,\infty)\times\sphere{n-1}$ with metric 
$$
g=dr^2+\Phi(r)^2\cdot g_{\sphere {n-1}}, \quad \Phi(r)\doteq s'_K(r)=\cosh(r\sqrt{\abs K}).
$$
The slice $\Sigma_r=\{r\}\times\sphere{n-1}$ is isometric to the sphere of radius $\cosh (r\sqrt{\abs K})$ and its mean curvature with respect to the normal $\nabla r$ is given by 
$$
H(r)=-\dfrac{\Phi'(r)}{\Phi(r)}=-\sqrt{\abs{K}}\tanh(r\sqrt{\abs K}).
$$
Given $H\in(-\sqrt{\abs K}, \sqrt{\abs K})$, we let $A=-\frac{1}{\sqrt{\abs K}}\tanh^{-1}(\frac{H}{\sqrt{\abs K}})$; we define
$$
\bar\Omega=[A,A+R]\times\sphere{n-1},
$$
and denote by $\Gamma_1$ the boundary component $\{A\}\times \sphere{n-1}$. One checks that $\Gamma_1$ has mean curvature $H$ with respect to the inner unit normal $N=\nabla r$ of $\bar\Omega$.

\medskip
{\bf Case 4.} $K<0, H=\pm\sqrt{\abs K}.$ These are the limiting cases of Case 3 as $H\to\pm\sqrt{\abs K}$.

\smallskip

With the above definitions, we can now state the following theorem.

\begin{theorem} Let $\lambda_1(\bar\Omega,\sigma)$ be the first eigenvalue of the problem
\begin{equation}\label{problembar}
\twosystem
{\Delta u=\lambda u\quad\text{on}\quad\bar\Omega}
{\derive uN=\sigma u \quad\text{on}\quad \Gamma_1, \quad\derive uN=0\quad\text{on}\quad \Gamma_2.}
\end{equation}
Then $\lambda_1(\bar\Omega,\sigma)=\lambda_1(R,\Theta,\sigma)$. 
\end{theorem}

\begin{proof} We let $\rho:\bar\Omega\to\reals$ be the distance of a point of $\bar\Omega$ to the component $\Gamma_1$ of $\bd\Omega$. From its definition, the cut-locus of $\Gamma_1$ in $\bar\Omega$ is empty, hence $\rho$ is $C^{\infty}-$smooth. Given the symmetries of $\bar\Omega$, the first eigenvalue of \eqref{problembar} is radial, and depends only on the distance to $\Gamma_1$, 
so it can be written $u=v\circ\rho$. Now
$$
\Delta u=-v''\circ\rho+(v'\circ\rho)\Delta\rho.
$$
As $\rho$ is smooth, we have
$
\Delta\rho=-\frac{\Theta'}{\Theta}\circ\rho,
$
which is $(n-1)$-times the mean curvature of the level set $\{\rho=r\}$. Computing the mean curvature of the level sets of $\bar\Omega$ one can check that in all of the above cases the formula holds with $\Theta(r)=(s'_K(r)-Hs_k(r))^{n-1}$, the weight function defined in \eqref{theta}. 
Ultimately one sees that the first eigenvalue of problem \eqref{problembar} coincides with the first eigenvalue of the problem on $[0,R]$:
\begin{equation}\label{barproblem}
\twosystem
{v''+\frac{\Theta'}{\Theta}v'+\lambda v=0}
{v'(0)=\sigma v(0), \quad v'(R)=0}
\end{equation}
and the assertion is proved. 
\end{proof}


\subsection{Proof of Theorem \ref{equalitycase}} In the cases at hand the model annulus $\bar\Omega$  is contained in a ball $\tilde\Omega$ of $M_K$, and moreover $\bd\tilde\Omega=\Gamma_1$. From Theorem A in \cite{K} ( we know that $\tilde R\geq R$ with equality if and only if $\Omega$ is isometric to $\tilde\Omega$. 
Now the first Robin eigenvalue of $\tilde\Omega$ is the first eigenvalue of the problem on $[0,\tilde R]$:
$$
\twosystem
{v''+\frac{\Theta'}{\Theta}v'+\lambda v=0}
{v'(0)=\sigma v(0), \quad v'(\tilde R)=0.}
$$
We compare this problem with problem \eqref{barproblem}, and as $\tilde R\geq R$ we see immediately from Lemma \ref{oned} that $\lambda_1(\bar\Omega,\sigma)\geq\lambda_1(\tilde\Omega,\sigma)$ with equality if and only if $\tilde R=R$. As $\lambda_1(\Omega,\sigma)
\geq \lambda_1(\bar\Omega,\sigma)$ we see that, a fortiori,  $\lambda_1(\Omega,\sigma)\geq \lambda_1(\tilde\Omega,\sigma)$ with equality iff $\tilde R=R$, that is, iff $\Omega$ is isometric to $\tilde\Omega$.




\end{document}